\documentclass[10pt]{amsart}   
\usepackage{amssymb, latexsym, amscd, graphics, epsfig, color}
\input xy
\input epsf

\newtheorem{theorem}{Theorem}[section]
\newtheorem{lemma}[theorem]{Lemma}
\newtheorem{proposition}[theorem]{Proposition}
\theoremstyle{definition}
\newtheorem{definition}[theorem]{Definition}
\newtheorem{corollary}[theorem]{Corollary}
\theoremstyle{remark}

\numberwithin{equation}{section}

\providecommand{\bs}{\backslash}
\providecommand{\ve}{\varepsilon}
\providecommand{\C}{\mathbb{C}}

\providecommand{\ba}{\backslash}
\providecommand{\be}{\beta}\providecommand{\e}{\epsilon}

\providecommand{\cC}{\mathcal{C}}
\providecommand{\de}{\delta}
\providecommand{\dG}{\delta_\Gamma}
\providecommand{\G}{\Gamma}

\providecommand{\g}{\gamma}
\providecommand{\La}{\Lambda}

\providecommand{\Q}{B}
\providecommand{\Pb}{{B}}
\providecommand{\q}{b}
\providecommand{\R}{\mathbb{R}}

\providecommand{\bH}{\mathbb{H}}
\providecommand{\Hn}{\mathbb{H}^n}

\providecommand{\Om}{\Omega}
\providecommand{\cO}{\mathcal{O}}
\providecommand{\tOm}{{\Om}}

\providecommand{\bms}{m^{\mathrm{BMS}}_\G}

\providecommand{\op}{\operatorname}
\providecommand{\br}{m^{\mathrm{BR}}_\G}

\providecommand{\PSL}{\mathop{\rm PSL}}

\providecommand{\T}{\operatorname{T}}

\providecommand{\GmodGa}{\Gamma \backslash G}
\providecommand{\br}{\mathbb R}

\providecommand{\pH}{\partial_\infty(\bH^n)}
\providecommand{\BMS}{\op{BMS}}
\providecommand{\BR}{\op{BR}}

\begin{document}
\title[Equidistribution]{On the distribution of orbits of geometrically finite
 hyperbolic groups on the boundary\\
(with appendix by Fran\c cois Maucourant)}

\author{Seonhee Lim }
\address{Department of Mathematics, Seoul National University, Seoul, 151-747, Korea}
\email{slim@snu.ac.kr,}

\author{Hee Oh}
\address{Mathematics department, Brown university, Providence, RI
and Korea Institute for Advanced Study, Seoul, Korea}

\email {heeoh@math.brown.edu}
\thanks{The authors are supported in part by Korean NRF 0409-20100101 and by 
NSF Grant 0629322 respectively}
\address{Universit\'e Rennes I, IRMAR, Campus de Beaulieu 35042 Rennes cedex -  France}
\email{francois.maucourant@univ-rennes1.fr}



\begin{abstract}
We investigate the distribution of orbits of a non-elementary discrete hyperbolic
subgroup $\G$ acting on $\Hn$ and its geometric boundary $\partial_\infty(\bH^n)$.
In particular, we show that if $\G$ admits a finite
Bowen-Margulis-Sullivan measure (for instance, if $\G$ is geometrically finite), then
every $\G$-orbit in $\pH$ is equidistributed with respect to the Patterson-Sullivan measure supported on the limit set $\Lambda(\Gamma)$.
The appendix by Maucourant is the extension of a part of his thesis where
 he obtains the same result as a simple application of Roblin's theorem.

Our approach is via establishing the equidistribution of
solvable flows on the unit tangent bundle of $\G\ba \bH^n$, which is of independent interest.
\end{abstract}

\maketitle
\section{Introduction}

Let $G$ be the group
of orientation preserving isometries
of the hyperbolic space $\bH^n$ and $
\G <G$ a torsion-free non-elementary (=not virtually abelian) discrete subgroup.
The action of $\G$ extends to $\overline{\bH^n}:=\bH^n \cup
\partial_\infty(\bH^n)$ where $\partial_\infty(\bH^n)$ denotes
the geometric boundary of $\bH^n$, and we define
the limit set $\Lambda(\G)$ as the set of accumulation points of a $\G$-orbit
in $\overline{\bH^n}$.
 


If we denote by $\delta_\G$ the critical exponent of $\G$,
then there exists a $\G$-invariant conformal density $\{\nu_x:x\in \bH^n\}$ of 
dimension $\delta_\G$ on $\Lambda(\G)$ by
 Patterson \cite{Patterson1976} for $n=2$ and Sullivan \cite{Sullivan1979} for $n$ general.
We consider the Bowen-Margulis-Sullivan measure $m^{\BMS}_\G$ on the unit tangent bundle
 $\op{T}^1(\Gamma\ba \bH^n)$ associated to the density $\{\nu_x\}$
(Def.~\ref{def:BMS}).  When the total mass $|m^{\BMS}_\G|$ finite,
 the geodesic flow is ergodic  on $\T^1(\G\ba \bH^n)$
\cite{Sullivan1979}.


For a subset $\Omega\subset \pH$ and $x
 \in \bH^n$,
we denote by $S_{x}(\Omega)\subset \bH^n$ the set of all points lying in geodesics emanating from $x$ toward $\Omega$, and by $B_T(x)\subset \bH^n$ the hyperbolic ball of radius $T$ centered at $x$.

Our main theorem is the following:
\begin{theorem}\label{thm:main}
Suppose that the total mass $|m_\G^{\BMS}|$ is finite.
Let $\Omega_1$ and $ \Omega_2$
be Borel subsets of $\partial_\infty(\bH^n)$ whose
boundaries are of zero Patterson-Sullivan measure.
Then, for any $x, y\in \bH^n$ and $\xi\in \pH$, as $T \to \infty$,
\begin{align*}\label{eqn:main}
\# \{  \gamma^{-1}(y)\in S_{x}(\Om_1)\cap B_T(x):
\;\; \g (\xi) \in \Om_2\} \sim \frac{
\nu_x (\Om_1)\nu_y (\Om_2)  }{\de_\G \cdot |\bms |}
\cdot  e^{\de_\G T}.
\end{align*}
\end{theorem}
If $\G$ is geometrically finite, that is,
if the unit neighborhood of
the convex core{\footnote{The \textit{convex core} $\cC_\G\subset \G\ba \bH^n$ is defined to be the minimal convex set which contains all geodesics connecting any two points in $\La(\G)$.
}} $\cC_\G$ has finite volume,
then $|m^{\BMS}_\G|<\infty$ \cite{Sullivan1984}.
However the above theorem is not restricted only to those groups as there are
 geometrically infinite
groups with $|m^{\BMS}_\G|<\infty$ (see \cite{Peigne2003}).

We remark that the assumption of
$|m^{\BMS}_\G|<\infty$ implies that the conformal density $\{\nu_x\}$ is determined uniquely
 up to homothety
(see \cite[Coro.1.8]{Roblin2003}).

\begin{figure}
\begin{center}
\includegraphics[width=9cm]{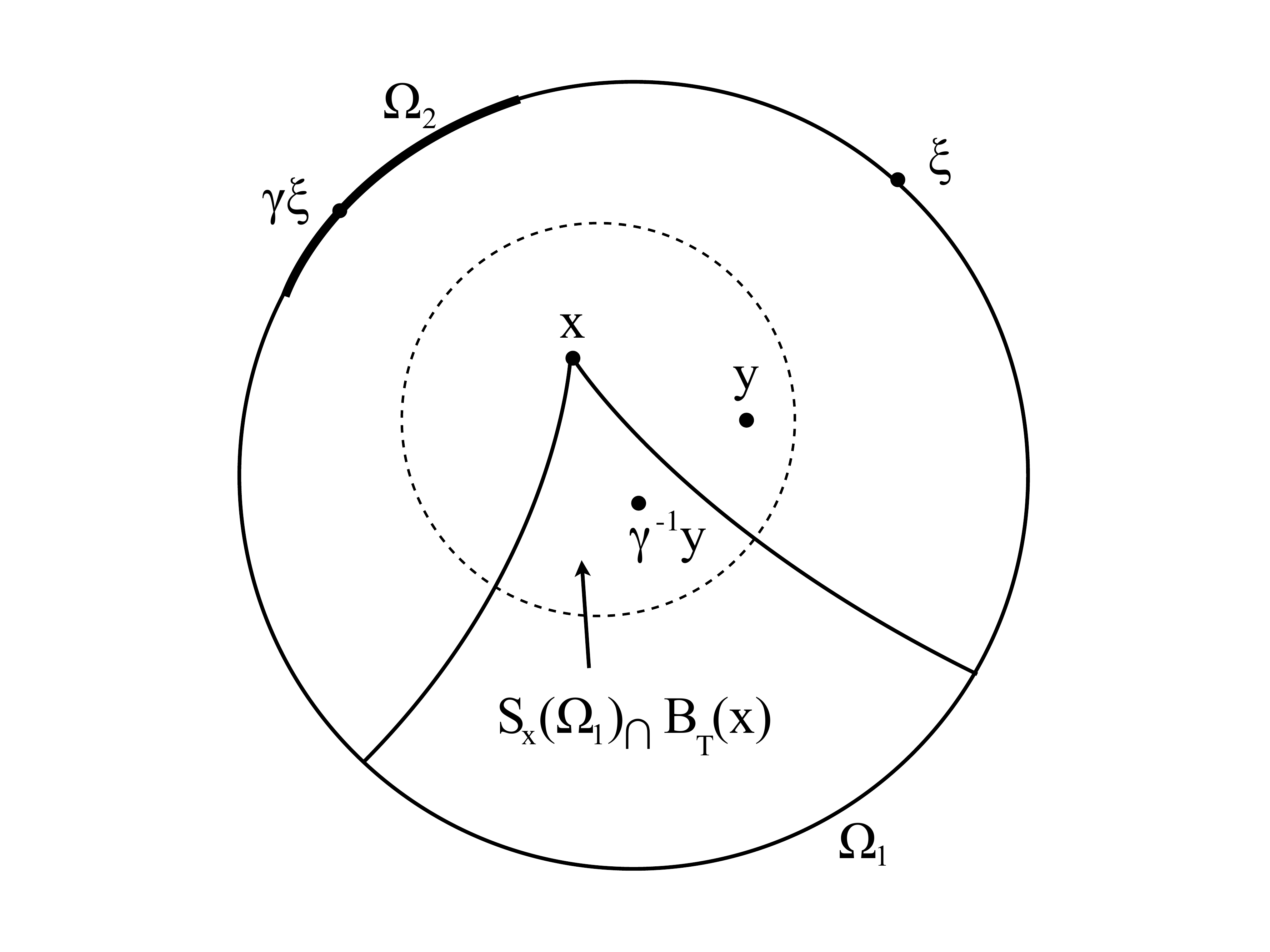}
\caption{Orbits of $\G$ on $\mathbb{H}^n \times \partial_\infty(\bH^n)$}
\end{center}
\end{figure}

\vspace{-2pt}

When $\Omega_1=\Omega_2=\partial_\infty(\bH^n)$, the above counting
problem is simply the non-Euclidean lattice point counting problem, and was solved
 by Lax and Phillips \cite{LaxPhillips} for geometrically finite groups with $\delta_\G >(n-1)/2$.
 Theorem \ref{thm:main} for $\Omega_2=\partial_\infty(\bH^n)$
 is due to Roblin \cite{Roblin2003}.
When $\G$ is a lattice,
 the same type of orbital counting result for  $\Omega_2=\partial_\infty(\bH^n)$
 was obtained in a much more
 general setting of Riemannian symmetric spaces
 (see \cite{Margulisthesis}, \cite{Bartels1982}, \cite{DukeRudnickSarnak1993}, \cite{EskinMcMullen1993}, etc.).
 Theorem \ref{thm:main} for general $\Omega_1,\Omega_2$
 was proved in \cite{GorodnikOh2007} for all lattices in semisimple Lie groups
 (see also \cite{GorodnikMaucourant} for the case when $\Omega_1=\partial_\infty(\bH^n)$).



We highlight Theorem \ref{thm:main} for the M\"obius transformation action of $\PSL_2(\C)$, that is,
the action on the extended complex plane $\widehat \C=\C\cup \{\infty\}$ by
$$\begin{pmatrix} a &b\\ c& d\end{pmatrix} (z)=\frac{az +b}{cz +d}$$
where $a,b,c,d\in \C$ with $ad-bc=1$ and $z\in \widehat \C$.
In the upper half-space model $\bH^3=\{(x,y,r): r>0\}$ of the hyperbolic $3$-space
with the metric $d=\frac{\sqrt{dx^2+dy^2+dr^2}}{r}$,
 the M\"obius transformations by elements of $\PSL_2(\C)$ give rise to all orientation preserving isometries
 of $\bH^3$.






For $g=\begin{pmatrix} a &b\\ c& d\end{pmatrix}\in \PSL_2(\C)$,
we have $$\cosh(d(g(j),j))=\frac{ |a|^2+|b|^2+|c|^2+|d|^2}{2} ,$$
where $j=(0,0,1)$.
Hence the following follows from Theorem \ref{thm:main}:
\begin{corollary}\label{thm:motivation}
Let $\G<\PSL_2(\C)$ be a non-elementary geometrically finite discrete subgroup.
For any Borel subset $\Om$ of $\widehat{\C}$ with $\nu_j(\partial(\Omega))=0$, we have,
as $T \to \infty$,
\[ \# \left\{ \begin{pmatrix} a &b\\ c& d\end{pmatrix}\in \G: |a|^2+|b|^2+|c|^2+|d|^2
<2 \cosh T, \;\;\frac{az+b}{cz+d} \in \Om \right\} \sim \frac{ |\nu_j| \cdot \nu_j(\Om)}{ \delta_\G \cdot |\bms|}\cdot  {e^{\de_\G T}} .\]
\end{corollary}

A similar result holds for the linear fractional transformation
action of non-virtually cyclic and finitely generated subgroups of $\PSL_2(\R)$ on $\widehat{\R}$.

\bigskip

After the submission, we were pointed out by the referee that in F. Maucourant's
thesis \cite{Mau}, Theorem \ref{thm:main} was already proved
 in the case when the sector is taken to be the whole ball (i.e.,
 $\Omega_1=\partial_\infty (\mathbb{H}^n)$) and that his approach which elegantly uses
 a theorem of Roblin \cite[Theorem 4.11]{Roblin2003} can be extended to obtain
  Theorem \ref{mmmm} of the Appendix.
 As Maucourant's result is not published, Maucourant agreed to write an appendix
 on his result.

Our approach is different from his, as we do not rely on
the aforementioned theorem of Roblin but on a recent result of Oh and Shah (see Theorem \ref{thm:OhShah}).
In section 2, we obtain the main ergodic  theorem which is the equidistribution of solvable flows
(Theorem \ref{thm:6.1}) which is of independent interest. In section 3, we relate the counting
function in Theorem \ref{thm:main} with an average over a solvable flow of a certain
function on $\op{T}^1(\G\ba \bH^n)$ (Lemma \ref{lem:thicken1}) and then apply the results in section 2 to
conclude Theorem \ref{thm:main}. Some computations such as Lemma \ref{aux}
are a bit tricky due to the fact that
 the Burger-Roblin measure $m_\G^{\BR}$ is not
an invariant measure in general.

 This approach of establishing
the equidistribution of $\G$-orbits on the boundary via the study of solvable flows on $\op{T}^1(\G\ba \bH^n)$
was first used in \cite{GorodnikOh2007}.

\bigskip
\noindent{\bf Acknowledgment:} We thank Thomas Roblin for useful comments.

\section{Equidistribution of solvable flows}\label{s:equidistribution}
For $x,y \in \Hn$ and $\xi\in \pH$, the \textit{Busemann function} $\be$ is defined as follows:
 \[ \beta_{\xi}(x,y) = \underset{t \to \infty}{\lim} \{ d(x, \xi_t)-d(y,\xi_t)\}. \]
where $\xi_t$ is a geodesic ray toward $\xi$.

For a unit tangent vector $u \in \op{T}^1(\Hn)$,
 we denote  by $\pi(u)$ the base point of $u$ and by $u^+$ (resp. $u^-$) the forward (resp. backward)
 endpoint of the geodesic determined by $u$.

Let $\G$ be a non-elementary discrete subgroup
of $G=\op{Isom}^+(\bH^n)$. Let  $\{\nu_x :x\in \bH^n\}$
denote a Patterson-Sullivan density for $\G$, i.e., each $\nu_x$ is a finite measure
supported on $\pH$ satisfying:
 for any $x,y\in \bH^n$, $\xi\in \partial_\infty(\bH^n)$ and $\gamma\in
\G$,
$$\gamma_*\nu_x=\nu_{\gamma x};\quad\text{and}\quad
 \frac{d\nu_y}{d\nu_x}(\xi)=e^{-\delta_\G \beta_{\xi} (y,x)}, $$
where $\gamma_*\nu_x(R)=\nu_x(\gamma^{-1}(R))$.

\begin{definition}\label{def:BMS}
{\rm The \textit{Bowen-Margulis-Sullivan measure} $\bms$ associated to $\{\nu_x\}$
 (\cite{Bowen1971}, \cite{Margulisthesis}, \cite{Sullivan1984})
is defined as the measure induced on $\op{T}^1(\G\ba \bH^n)$ of the following $\G$-invariant  measure $\widetilde{m}^{\mathrm{BMS}}$ on $\op{T}^1(\Hn)$:
\[ d \widetilde{m}^{\mathrm{BMS}} (u) =  e^{\dG \be_{u^+}(x, \pi(u))}e^{\dG \be_{u^-}(x, \pi(u))} d\nu_x(u^+) d\nu_x(u^-)dt .\]}
\end{definition}

We  denote by $\{m_x:x\in \bH^n\}$ a
$G$-invariant conformal
density of dimension $n-1$, which is unique up to homothety.

\begin{definition}{\rm The \textit{Burger-Roblin measure} $\br$ associated to
 $\{\nu_x\}$ and $\{m_x\}$ (\cite{Burger1990}, \cite{Roblin2003}) is defined as the measure induced on $\op{T}^1(\G \bs \Hn)$ of the following $\G$-invariant measure $\widetilde{m}^{\mathrm{BR}}$ on $\op{T}^1(\Hn)$:
\[ d \widetilde{m}^{\mathrm{BR}} (u) =  e^{(n-1) \be_{u^+}(x, \pi(u))}e^{\dG \be_{u^-}(x, \pi(u))} dm_x(u^+) d\nu_x(u^-) dt .\] }
\end{definition}

The
measure $\widetilde{m}^{\mathrm{BR}}$ is supported on the set of unit tangent vectors $u$
such that $u^-$ belongs to the limit set $\La_\G$.
\bigskip

We fix  $x\in \bH^n$ and
$\xi\in \pH$ in the rest of this section.  Let $K$ be the stabilizer of $x$ in $G$
 and $P$ denote the stabilizer of $\xi\in \pH$.
The subgroup $P$ is a minimal parabolic subgroup of $G$ and is the normalizer of its
 unipotent radical $N$.
Without loss of generality, we may assume that $m_x$ is the probability measure.

 Denote by $X_0\in \op{T}^1(\bH^n)$ the unit vector based at $x$ such that $X_0^-=\xi$.
 We set $$\xi_x:=X_0^+ .$$
 Setting $A=\{a_t:=\exp(tX_0): t\in \R \} ,$
we have   (cf. \cite[Lem. 4.1]{GorodnikOh2007})
\begin{itemize}
\item $G=KA^+K$  where
  $A^+:=\{a_t:t\ge 0\}$;
\item $P=MAN$ where $M$ is the centralizer of $A$ in $K$ and $M=K\cap P$;
\item $N$ is the expanding horospherical subgroup of $G$ with respect to $A^+$, i.e.,
 $N=\{g\in G: a_tga_{-t}\to e\quad \text{as $t\to \infty$}\}$.
 \end{itemize}

The above Cartan decomposition $G=KA^+K$
 says that for any $g \in G$, there exists a unique element $a \in A^+$ such that $g = k_1 a k_2$, for $k_1, k_2 \in K$. Moreover,
 $k_1a k_2 = k'_1 a k'_2$
implies that $k_1 = k'_1 m $ and $k_2 = m^{-1}k'_2$ for some $m\in M$.

We may identify $G/K$ with $\bH^n$
where $gK$ corresponds to $g(x)$ and $G/M$ with $\op{T}^1(\bH^n)$ where $gM$ corresponds to
$g(X_0)$.

Let $B_0$ be the maximal split
solvable subgroup of $G$ given by $$B_0=AN .$$

 For $T>0$ and a subset $\Omega\subset K$ with $\Omega M=\Omega$, set
\[ B_0(T, \Omega) : =  B_0 \cap \Omega A_T^+ K  \]
where $A_T^+:=\{a_t: 0\le t\le T\}$.
Our aim in this section is to prove an equidistribution of $B_0(T,\Omega)$
on $\op{T}^1(\G\ba \bH^n)$: Theorem \ref{thm:6.1}.

The following is the main ergodic ingredient we use.
\begin{theorem}\cite{OhShahGFH}\label{thm:OhShah} Suppose that
$|m_\G^{\BMS}|<\infty$.
Let $\Om$ be a Borel subset of $K$ with $\Om M= \Om$ and
with $\nu_x(\partial(\Om(\xi_x)))=0$.
 For any $\varphi \in C_c(\G \bs G)^M$,
$$ e^{(n-1-\de_\G)t}  \int_{s \in \Om/M} \varphi(sa_t) dm_x(s) \sim
\frac{\nu_x(\Om(\xi_x))}{|\bms |}\cdot  m^{\BR}_\G(\varphi)  \quad\text{ as $t \to  +\infty$}.$$
\end{theorem}

By the Iwasawa decomposition $G=ANK$, the map
\[ K \longrightarrow B_0\ba G : k \mapsto B_0k \]
is a diffeomorphism, say, $\iota$.
 Let $N^-$ be the contracting horospherical subgroup of $G$
 with respect to $A^+$: $N^-=\{g\in G: a_{-t}ga_{t}\to e\quad \text{as $t\to \infty$}\}$.

 The map $M \times N^-\to
B_0\ba G$, $mn\mapsto B_0 mn$, composed with $\iota^{-1}$,
 is an
$M$-equivariant map $M \times N^- \to K$ which is a diffeomorphism onto
its image, which is a Zariski open subset. Let $S$ be the image of
$\{e \}\times N^-  $ under this map. We note that the complement of $M\ba MS$
in $M\ba K$ is a point.



\begin{lemma}\label{go} Let $s\in S$.
 If $V\subset S$ is a neighborhood of $s$  and $S_0$ is a compact subset of $S$,
there exists $C=C(S_0)>1$ such that for any $m\in M$,
$$  MS_0 m \subset  MV s^{-1} m a_{-t} \quad\text{for all $t>C$.}$$

\end{lemma}
\begin{proof}
Since $e\in V s^{-1}$,
the conjugation by $a_t$ expands $Vs^{-1}\subset S$ by the factor of $e^t$, and hence
we can find  $C>1$
such that
$$S_0\subset a_t Vs^{-1}a_{-t}$$
for all $t>C$.
Hence
\begin{align*}  B_0MS_0m
 \subset B_0M  a_{t}Vs^{-1}a_{-t}m
=B_0MVs^{-1} ma_{-t} 
\end{align*}
as $a_t\in B_0$.

By the uniqueness of the decomposition $G=B_0K$, we have the desired inclusion.
\end{proof}




We denote by $dh$ the Haar measure on $G$
such that for $h=k_1a_tk_2\in KA^+K$,
$$dh=2^{n-1}(\sinh t\cosh t)^{(n-1)/2} dk_1 dtdk_2$$
where $dk$ denotes the probability Haar measure on $K$.

We denote by $\rho_\ell$ the left-invariant Haar measure on $B_0$
given by the relation:
 $$dh=d\rho_{\ell}(b) dk$$
 where $h=bk\in B_0K$.


In the rest of this section, we assume that
$|m_\G^{\BMS}|<\infty$.

The following lemma is a special case of \cite[Prop. 3.1]{Roblin2000}:
\begin{lemma}\label{rud}
Any sphere centered at $\xi_0\in \Lambda(\G)$ has
measure zero with respect to $\nu_x$.
\end{lemma}





\begin{proposition}\label{prop:6.6}  Let $V$ be an open
neighborhood of $e$ in $K$ such that $MV=V$.
Let $\Om$ be a Borel subset of $K$ with $\Om M= \Om$ and
with $\nu_x(\partial(\Om(\xi_x)))=0$. Then for any $\psi \in C_c(\GmodGa)$,
$$  \int_V \int_{B_0(T, \Om)} \psi (bk) d\rho_{\ell} (b) dk \sim
 \frac{e^{\de_\G T} \nu_x(\Om( \xi_x))}{\de_\G\cdot |\bms |}\cdot
\br(\psi *\chi_V), $$ as $T \to \infty,$ where $ \psi
* \chi_{V}(h) =\int_{k\in V} \psi(h k ) dk $
and $B_0(T, \Omega) : =  B_0 \cap \Omega A_T^+ K$.
\end{proposition}
\begin{proof}
Note that
\begin{align*}
B_0(T, \Omega)V
&= B_0V \cap \Om A_T^+K   \\
&=\{k_1 a_{t} k_2 : k_1 \in \Om_{k_2}(t),  0<t<T  ,k_2 \in K \},
\end{align*}
where  $\Om_{k_2}(t) =  \Omega \cap  B_0Vk_2^{-1} a_{-t}$.

Setting $\Xi(t)=2^{n-1}(\sinh t\cosh t)^{(n-1)/2}$, we have \begin{align*}
&\int_{B_0(T, \Om)} \psi (bk)  d\rho_\ell (b)dk\\&=
 \int_{h\in B_0(T, \Om)} \psi (h) dh
\\  &=\int_{ k_1a_{t}k_2 \in B_0(T, \Om)}
 \psi (k_1a_{t}k_2) \Xi(t) dk_2dtdk_1\\
 &= \int_{k_2 \in K} \int_{0<t<T} \int_{k_1 \in \Om_{k_2}(t)}
 \psi (k_1a_{t}k_2) \Xi(t)dk_2dtdk_1.
 \end{align*}

 Set for $m\in M$,
 $$\Om_m:=  \Om  \cap  MS m^{-1}  $$
where $S$ is the image of $N^-$ in $K$.
We note that since $S\subset M\ba K$ is an open Zariski dense subset whose complement is a point and $\nu_x$
  is atom free,
$\nu_x(\Om)=\nu_x(\Om_m)$ and
$\nu_x(\partial(\Om))=\nu_x(\partial(\Om_m))$.
Write $V=MV_0$ for $V_0\subset S$.
Let $k_2=ms\in MS$ with $s\in V_0$.

By Lemma \ref{rud},
 for any fixed $\e>0$, we can
 take a compact subset $S_\e \subset S$
 such that
 $\nu_x(\Om (\xi_x) - S_\e(\xi_x))<\e$ and $\nu_x(\partial(S_\e(\xi_x)))=0$.
If we set $\Om_m(S_\e):=\Om  \cap MS_\e m^{-1}$, then
$\nu_x(\Om_m(\xi_x) - \Om_m(S_\e)(\xi_x))<\e$
and
$\nu_x(\partial(\Om_m(S_\e) (\xi_x)))=0$ since
$\partial(\Om_m(S_\e)(\xi_x))\subset \partial(\Om(\xi_x))\cup \partial(S_\e(\xi_x)))$.

 By Lemma \ref{go},
 there exists $C_\e>1$ such that for all $t>C_\e$,
$$ \Om_m(S_\e) \subset\Om_{ms}( t).$$
On the other hand, as $a_t\in B_0$,
$$\Om_{ms}(t)\subset \Om \cap B_0MVs^{-1}m^{-1}a_{-t}
=\Om \cap B_0M (a_t Vs^{-1}a_{-t}) m^{-1}
\subset \Om\cap MSm^{-1}=\Om_m. $$

Without loss of generality we assume below that $\psi$ is non-negative.
Hence for all $t>C_\e$, $$
\int_{k_1\in \Om_m(S_\e) } \psi (k_1a_tsm )  dk_1\le  \int_{k\in \Om_{ms}(t) }  \psi (k_1a_tms) dk_1
 \le\int_{k_1\in \Om_m } \psi (k_1a_tms)  dk_1
. $$

Note that by applying Theorem \ref{thm:OhShah}
\begin{align*} &\int_{k_1\in \Om_{m}(S_\e)}  \psi (k_1a_tms) dk_1\\
&=  \int_{s\in \Om_{m}(S_\e)/M} \int_{m_1\in M}  \psi (s a_t m_1ms) dm_1dm_x(s)
\\ &= \int_{s\in \Om_{m}(S_\e)/M} \psi_{ms} (s a_t ) dm_x(s)\\
& \sim   e ^{-(n-1-\delta_\G)t} \frac{1}{ |m^{\BMS}_\G|} m^{\BR}_\G(\psi_{{ms}}) \nu_x
 ( \Om_m(S_\e) (\xi_x))
\end{align*}
where $\psi_{ms}(h):=\int_{m_1\in M}\psi(h m_1ms)dm_1$.

Hence
\begin{align*}
&\liminf_t e^{(n-1-\delta_\G )t} \int_{k_1\in \Om_{ms}(t) }
 \psi(k_1a_tms) dk_1 \\
&\ge \liminf e^{(n-1-\delta_\G )t} \int_{k_1\in \Om_{m}(S_\e)}
\psi(k_1a_tms) dk_1\\
&= \frac{1}{|m^{\BMS}_\G|} m^{\BR}_\G(\psi_{{ms}}) \nu_x
 ( \Om_m(S_\e)(\xi_x) )\\
 &\ge \frac{1}{|m^{\BMS}_\G|} m^{\BR}_\G(\psi_{{ms}}) (\nu_x
 ( \Om_m(\xi_x))-\e )
\end{align*}
and similarly
$$ \limsup_t e^{(n-1-\delta_\G )t} \int_{k_1\in \Om_{ms}(t) }
 \psi(k_1a_tms) dk_1 \le
\frac{1}{ |m^{\BMS}_\G|} m^{\BR}_\G(\psi_{{ms}}) \nu_x
( ( \Om_m (\xi_x))  +\e).$$

As $\e>0$ is arbitrary and $\Om_m(\xi_x)=\Om(\xi_x)$,
we deduce
 $$ \lim_{t \to \infty} e^{(n-1-\delta_\G )t} \int_{k\in \Om_{ms}(t) }  \psi(k_1a_tms) dk_1 =
\frac{1}{ |m^{\BMS}_\G|} m^{\BR}_\G(\psi_{{ms}}) \nu_x
 ( \Om(\xi_x) ) .$$

  Using that $\Xi(t)\sim_t e^{(n-1)t}$, we obtain that
     for any $ms\in MV_0$, as $T\to \infty$,
  \begin{align*}
  & \int_{A_{T}^+}\int_{\Om_{ms}(t)}  \psi (k_1a_tms) \Xi(t) dk_1dt
  \\ &\sim \frac{ e^{\delta_\G T} }{\de_\G\cdot |m^{\BMS}_\G|}
 m^{\BR}_\G( \psi_{ms}) \nu_x(\Om(\xi_x))
   .\end{align*}

 Now  for $ms\notin MV_0$,
 we claim that  $$\limsup_T e^{-\delta_\G T}
 \int_{A_{T}^+}\int_{\Om_{ms,t}} \psi (k_1a_tms) \Xi(t) dk_1dt =0 .$$

Consider the set
$$\Om_{ms,t}^c:=\Omega\cap B_0(MS- MV_0) s^{-1}m^{-1}a_{-t} .
$$
As $s\in MS-MV_0$,
we have by the previous case that
 $$\lim_T e^{-\delta_\G T}\int_{A_{T}^+(C)}\int_{\Om_{ms,t}^c}
\psi (k_1a_tms) \Xi(t) dk_1dt
  =\frac{1}{\de_\G\cdot |m^{\BMS}_\G|} m^{\BR}_\G( \psi _{ms}) \nu_x(\Om_m(\xi_x))
   .$$

   Since $\Om_{ms,t}^c\subset \Om_m$
   and $$\lim_T e^{-\delta_\G T}\int_{A_{T}^+(C)}\int_{\Om_{m}}
 \psi (k_1a_tms) \Xi(t) dk_1dt
  =\frac{1}{ \de_\G\cdot |m^{\BMS}_\G|} m^{\BR}_\G( \psi _{ms}) \nu_x(\Om_m(\xi_x)),
   $$
   the claim follows.

  Since the image of $S$ is an open Zariski dense subset of $M\ba K$,
   we may replace $K$ by $M S$ in the integration over $K$ and hence
\begin{align*}
 & \int_{k_2\in K} \int_{A_{T}^+}\int_{k_1\in \Om_{k_2}(t)}
\psi(k_1a_tms) \Xi(t) dk_1dt
 dk_2 \\
& \sim \int_{ms\in MV_0} \int_{A_{T}^+}\int_{k_1\in \Om_{k_2}(t)}
\psi (k_1a_tms) \Xi(t) dk_1dt
 dk_2
\\ &
 \sim \frac{ e^{\delta_\G T}}{\de_\G\cdot |m^{\BMS}_\G|} m^{\BR}_\G
(\psi *\chi_{V}
 ) \nu_x(\Om(\xi_x)).
  \end{align*}
This completes the proof of Proposition~\ref{prop:6.6}.
\end{proof}


\begin{theorem}\label{thm:6.1}
Let $\Om$ be a
Borel subset of $K/M$ with $\nu_x(\partial(\Omega(\xi_x)))=0$.
Then for any $\psi
\in C_c(\G \bs G)$, as $T \to \infty$,
\[
\int_{B_0(T,\Om)} \psi (b) d\rho_\ell(b)  \sim
\frac{ e^{\de_\G T}}{\de_\G} \cdot \frac{\nu_x(\Om(\xi_x))}{|\bms |}\cdot
m_\G^{\BR}(\psi) . \]
\end{theorem}
\begin{proof} Let $V_\e$ be an $\e$-neighborhood of $e$ in $K$
 such that $V_\e M=V_\e$.
   For any $\psi \in C_c(\G \bs G)^M$ and $\e>0$, define functions $\psi_\ve^\pm \in C_c(\G\ba G)^M$
as follows:
\[ \psi_\ve^+(h) := \underset{k \in V_\ve}{\sup} \psi(hk), \;\text{and}
\;\; \psi_\ve^-(h) := \underset{k \in V_\ve}{\inf} \psi(hk).\]
 Let $\eta>0$.
  By the uniform continuity of $\psi$ and the $M$-invariance,
    there exists $\e=\e(\eta)$
such that
 $|\psi_\ve^+ (h) - \psi_\e^-(h)|< \eta  $ for all $h\in G$.

Without loss of generality we may assume $\psi \ge 0$. Note that,
  by applying Proposition~\ref{prop:6.6},
  \begin{align*}
 &\limsup_T e^{-\delta_\G T}\int_{B_0( T,\Om)} \psi (b) d\rho_\ell (b) \\&\le
 \limsup e^{-\delta_\G T} \op{Vol}(V_\e)^{-1}
 \int_{k\in V_\e} \int_{B_0( T,\Om)} \psi_\e^+(bk) d\rho_\ell(b) dk
 \\ & = \op{Vol}(V_\e)^{-1}
 \frac{1}{\delta_\G   |m^{\BMS}_\G|} m^{\BR}_\G(\psi_{\e}^+*\chi_{V_\e} )\nu_x(\Om(\xi_x)) .
 \end{align*}
Similarly
 \begin{align*}
&\liminf_T  e^{-\delta_\G T} \int_{B_0( T,\Om)} \psi (b) d\rho_\ell (b) \\
&\ge \liminf_T  e^{-\delta_\G T} \op{Vol}(V_\e)^{-1}
 \int_{k\in V_\e} \int_{B_0( T,\Om)} \psi_\e^-(bk) d\rho_\ell(b) dk
 \\ & = \op{Vol}(V_\e)^{-1}
 \frac{1}{\delta_\G   |m^{\BMS}_\G|} m^{\BR}_\G(\psi_{\e}^-*\chi_{V_\e} )\nu_x(\Om(\xi_x)) .
 \end{align*}

Since
$$ m^{\BR}_\G(\psi_{\e}^\pm *\chi_{V_\e}) = m^{\BR}_\G(\psi)\op{Vol}(V_\e) +O(\eta),$$
we deduce that
$$
\limsup_T e^{-\delta_\G T}\int_{B_0( T,\Om)} \psi (b) d\rho_\ell (b)
=\frac{1}{\delta_\G   |m^{\BMS}_\G|} m^{\BR}_\G(\psi )\nu_x(\Om(\xi_x))
+O(\eta)$$
and
$$
\liminf_T e^{-\delta_\G T}\int_{B_0( T,\Om)} \psi (b) d\rho_\ell (b)
=\frac{1}{\delta_\G   |m^{\BMS}_\G|} m^{\BR}_\G(\psi )\nu_x(\Om(\xi_x))
+O(\eta).$$
As $\eta>0$ is arbitrary,
this proves the claim.
\end{proof}

\section{Proof of Theorem \ref{thm:main}}\label{s:groups}
Fix $x\in \bH^n$ and $\xi\in\pH$. We keep the same
notation from the previous section. Let $y\in \bH^n$ and choose $g\in G$ such that $g(x)=y$.

For a subset $W$ of $G$, we denote by
$W^g$ the conjugate $gWg^{-1}$. Note that $K^g$ is the stabilizer of $y$
and that $B:=B_0^g$ stabilizes $g(\xi)=g(X_0^-)$.

 For $\hat \Om_1, \hat \Om_2 \subset \pH$, we set
  $$\Om_1 :=\{k\in  K/M: k\xi_x \in \hat\Om_1\},\quad
  \Om_2 :=\{k\in K^g/M^g: k(g(\xi))\in \hat \Om_2 \}$$
  so that
  $\hat \Om_1=\Om_1(\xi_x)$ and $\hat \Om_2=\Om_2(g(\xi))$.
We assume that the boundaries of $\hat\Om_i$ have measure zero with respect to the Patterson-Sullivan density.

In this notation, we have
 $$\{z\in S_{x}(\hat \Omega_1), d(z, x)<T\} =
\Omega_1A^+_T(x) $$ and hence the condition
$\gamma^{-1} y \in  S_{x, T}(\hat \Omega_1)$
becomes $\gamma \in g K A^-_T\Omega_1^{-1}$.
And $\gamma (\xi)\in \hat \Om_2$ is equivalent to
$\gamma g^{-1}(g(\xi)) \in \Om_2 (g(\xi))$ and hence to
$\gamma g^{-1} \in \Om_2 B$.

For $h\in G$, we write $$h=h_{K^g}h_B g$$
where $h_{K^g}\in K^g$ and $h_B\in B$ are uniquely determined.

 Hence
setting
 $$B(T,\Om_1)=B\cap g\Om_1A^+_TK g^{-1},$$
 the number we want to count is the following:
\begin{align*}
N_T(\hat \Om_1,
\hat \Om_2) & := \# \{ \g \in \G : \g (g(\xi)) \in \hat \Om_2,
 \g^{-1}(y) \in S_x(\hat \Om_1) \} \\
&=\# \G \cap g K A^-_T\Omega_1^{-1}\cap \Omega_2B g\\
&=\{\gamma\in \G: \gamma_{K^g}\in \Om_2,\gamma_{B}^{-1}\in B(T,\Om_1)\}
.\end{align*}

Let $V_\e$ be an $\e$-neighborhood of $e$ in $K$
 such that $M V_\e M=V_\e$.

For the $\e$-neighborhood $A_\ve=\{a_t: |t|<\e\}$ of $e$ in $A$,
by the strong wavefront Lemma (see \cite{GorodnikShahOhIsrael} or \cite{GorodnikOh2007})
 there exists a symmetric neighborhood $\cO_\ve'$ of $e$ in $G$ and $C>1$
 such that for all $k\in K$ and  all $t>C$,
\begin{equation}\label{st}
g k a_t Kg^{-1} \cO_\ve' \subset g k V_\e a_t A_\e Kg^{-1}.
 \end{equation}

Choose a symmetric neighborhood $\tilde V_\e
\subset V_\e$ so that
\begin{equation}
\label{ve}
\op{Vol}(V_\e^+ - V_\e^-) <\eta \op{Vol}(V_\e)
\end{equation}
where $V_\e^+:=V_\e \tilde V_\e$ and $V_\e^-:=\cap_{u\in \tilde V_\e} V_\e u$.

We may assume without loss of generality that $\cO_\e'$
satisfies
$$a_t n k (g^{-1}\cO_\e' g)\subset a_t A_\e N k \tilde V_\e$$
for all $a_tnk\in ANK$.


We set $$\cO_\e := \cO_\e ' \cap \Pb$$
and note that $\cO_\e^{-1}=\cO_\e$.


Fix $\eta>0$. Then there exists $0<\e(\eta)<\eta$ such that for all $0<\e<\e(\eta)$,
 \begin{equation}\label{ueone}
\nu_x(\Om_{1,\e}^+ (\xi_x) -\Om_{1,\e}^-(\xi_x) )<\eta \end{equation}
where $\Om_{1,\e}^+=\Om_{1}V_\e^+$ and $\Om_{1,\e}^-=\cap_{k\in V_\e^-} \Om_1 k$.
 This is possible
since the boundary of $\hat \Om_1$ has measure zero
 with respect to $\nu_x$.

Similarly, we may assume that
 \begin{equation}\label{ue}
\nu_y(\Om_{2,\e}^+ (g(\xi)) -\Om_{2,\e}^-(g(\xi)) )<\eta \end{equation}
where $\Om_{2,\e}^+=\Om_{2}U_\e^+$ and $\Om_{2,\e}^-=\cap_{k\in U_\e^-} \Om_2 k$
where $U_\e^{\pm}:=g V_\e^{\pm} g^{-1}$.
We also set $U_\e=gV_\e g^{-1}$.

We choose $ \phi^{\pm}_\e=\phi_{\Om_2,\e}^{\pm} \in C_c(K^g)^{M^g}$ such that
$0\le \phi_{\Om_2,\e }^{-} \le \phi_{\Om_2,\e }^{+} \le 1$, $\phi_\e^+(k)=1$ for $k\in \Om_2$,
$\phi^+_\e(k)=0$ for $k\notin \Om_{2}U_\e$,
$\phi^-_\e(k)=1$ for $k\in \cap_{u\in U_\e} \Om_{2}u$, and
$\phi^-_\e(k)=0$ for $k\notin \Om_{2}$.

We denote by $\rho$ the left invariant Haar measure
on $B$ given by: for $\psi\in C_c(B)$,
$$\int_B \psi(b) d\rho(b):=\int_{b_0\in B_0}\psi(g^{-1} b_0g) d\rho_{\ell}(b_0).$$

Choosing a non-negative function $\psi_\e\in C_c(B)$ supported on $\cO_\e$ and
with $\int_B \psi_\e (b)d\rho(b)=1$, we
 define a function $f^{\pm}_{\Om_2,\eta}$ on $G=K^gBg$
 by $$f^{\pm}_{\Om_2, \eta}(h)  =
\phi_{\Om_2,\e(\eta)}^{\pm}(h_{K^g}) \psi_{\e(\eta)} (h_B)$$
where $h=h_{K^g}h_B g\in G$ with
  $h_{K}\in K^g$ and $h_B\in B$ uniquely determined and $\e=\e(\eta)$.
Define
$$F_{\Om_2,\eta}^{\pm} (h) = \underset{\g
\in \G}{\sum} f^\pm_{\Om_2,\eta}(\g h),$$ which is an integrable function defined on
$\G \bs G$.


We set
$$B_0^C(T,\Om_1):=B_0\cap \Omega_1A_T^+(C)K; $$
$$B^C(T,\Om_1):= B\cap g \Omega_1A_T^+(C)K g^{-1};$$
$$N_T^C(\hat \Om_1, \hat \Om_2) :=\# \G \cap
gK A^-_T(C)\Omega_1^{-1} \cap \Omega_2Bg$$
where $A^-_T(C)=\{a_{-t}: C<t<T\}$ and $A^+_T(C)=\{a_t:C<t<T\}$.
When $C=0$, we simply omit the superscript $0$ from the above notation.

Note that
$$N_T^C(\hat \Om_1, \hat \Om_2)=\{\gamma\in \G:\gamma_{K^g}\in \Om_2,
\gamma_B^{-1}\in B^C(T,\Om_1)\}.$$

\begin{lemma}\label{lem:thicken1} Let $C>1$ be taken so that
\eqref{st} holds.
For any
$T>1$ and small $\eta>0$, we have\begin{enumerate}
\item  \[  N_T^C(\hat \Om_1, \hat \Om_2) \leq
\int_{B_0({T+\ve}, \Om^+_{1,\e})} F_{\Om_2,\eta}^+ (b_0)
d\rho_{\ell}(b_0) ; \]
\item \[\int_{B_0^C({T-\ve}, \Om^-_{1,\e})} F_{\Om_2,\eta}^-(b_0)
d\rho_{\ell}(b_0) \leq   N_T(\hat \Om_1, \hat \Om_2)\]
\end{enumerate}
where $\Om^+_{1,\e}= \Om_1 V_\e $ and
$\Om^-_{1,\e}=\cap_{k\in V_{\epsilon}} \Om_1 k$
 and $\e=\e(\eta)$.
\end{lemma}
\begin{proof}
 For simplicity, we set $F^\pm:=F^\pm_{\Om_2,\eta}$ and $\Om_1^{\pm}:=\Om_{1,\e}^\pm$.
 We have
\begin{align*}&\int_{B_0({T+\ve}, \Om^+_{1,\e})} F^+ (b_0)
d\rho_{\ell}(b_0)\\&=
\int_{B(T+\e, \tOm_1^+)} F^+(g^{-1}bg) d\rho(b)  \\ &\ge
 \int_{B(T+\e, \tOm_1^+)}
\underset{\g \in \G}{\sum}
 \chi_{\Om_2}(\g_{K^g})\psi_\e (\g_\Pb b)d\rho(b)\\
&= \underset{\g \in \G, \; \g_{K^g} \in \Om_2}{\sum}
 \int_{\gamma_BB(T+\e, \tOm_1^+)
 \cap \cO_\e}\psi_\e(b) d\rho(b) \\
\end{align*}
since $\rho$ is left-invariant.
Since we have chosen $\cO_\ve$ so that
$ B^C(T, \tOm_1) \cO_\e\subset
B(T+\e, \tOm_1^+),$
 for any $\g\in \G$ such that $\g_{\Pb}^{-1} \in B^C(T, \tOm_1)$,
\[  B(T+\e, \tOm_1^+) \cap \g_\Pb^{-1}\cO_\e
=\g_\Pb^{-1} \cO_\ve , \]
and hence
$$\int_{\gamma_B \Pb^C(T+\e,\tOm_1^{+}) \cap \cO_\e}\psi_\e(b)
 d\rho(b) =\int_{ \cO_\e}\psi_\e(b)
 d\rho(b)=1 .$$

It follows that
\begin{align*}
\int_{B(T+\e, \tOm_1^+)} F^+(b) d\rho_{\ell}(b) & \geq
 \# \{ \g \in \G : \g_{K^g} \in \Om_2, \g_\Pb^{-1} \in B^C(T, \tOm_1)\} \\
&=N_T^C(\hat \Om_1, \hat \Om_2) .
\end{align*}
Similarly, we have
\begin{align*}&
\int_{\Pb_0^C(T-\e,\tOm_1^{-})} F^-(b_0) d\rho_{\ell}(b_0) \\
&=\int_{\Pb^C(T-\e,\tOm_1^{-})} F^-(g^{-1}bg) d\rho(b)
\\&\le
\int_{\Pb^C(T-\e,\tOm_1^{-})}
\underset{\g \in \G}{\sum} \chi_{\Om_2}(\g_{K^g})
\psi_\e (\g_\Pb b)d\rho(b)\\
&= \underset{\g \in \G, \; \g_{K^g} \in \Om_2}{\sum}\int_{\gamma_B \Pb^C(T-\e,\tOm_1^{-}) \cap \cO_\e}\psi_\e(b)
 d\rho(b).
\end{align*}
Since $\tOm_1^- V_\e\subset \tOm_1$, we have
\[\Pb^C(T-\e,\tOm_1^{-}) \cO_\e \subset \Pb(T,\tOm_1). \]
Therefore for $\g \in \G$ such that $\g_B^{-1}
 \notin \Pb (T,\tOm_1)$, we
have $\rho_{\ell}(\Pb^C(T-\e,\tOm_1^{-})
\cap \g_{\Pb}^{-1} \cO_\ve )=0.$

 Hence it follows
that
\[ \int_{\Pb^C(T-\e,\tOm_1^{-})} F(b)
 d\rho(b) \leq \# \{\g \in \G : \g_{K^g}\in \Om_2,
\g_\Pb^{-1} \in \Pb(T, \tOm_1)
 \}  = N_T(\hat \Om_1, \hat \Om_2) . \]

\end{proof}

\begin{lemma}\label{r} Let $k\in K$ and $k_2\in K^g$.
 Writing $k^{-1}k_2g=a_rnk_0\in ANK$,
we have $$r=\beta_{k\xi}(y, x) .$$
\end{lemma}
\begin{proof}
Since $\xi=\lim_{t\to\infty}a_{-t}x$, we compute that
\begin{align*}&\beta_{k\xi}(y, x)=\beta_{k\xi}(k_2y, x)\\
&= \beta_{\xi}(k^{-1} k_2 y,x)\\ &=
\lim_{t\to \infty}d(a_rnk_0 x, a_{-t}x) -t\\
&=\lim_{t\to \infty}d(a_r(a_t na_{-t}) a_tk_0 x, x) -t
=r .
\end{align*}
\end{proof}

For simplicity, we set $F^{\pm}_\eta:= F_{\Om_2,\eta}^{\pm}$
and $f^{\pm}_\eta:= f_{\Om_2,\eta}^{\pm}$.

\begin{lemma} \label{aux} We have
$$ \limsup_\eta  m_\G^{\BR}(F^+_{\eta}) =\liminf_\eta  m_\G^{\BR}(F_{\eta}^{-})=
  \nu_y({\hat \Om_2}) .$$
\end{lemma}
\begin{proof}
We use the formula for $\tilde m^{\BR}$: for any $\Psi\in C_c(G)^M$,
$$\widetilde{m}^{\mathrm{BR}}(\Psi)
 =\int_{KAN}  \Psi(k a_r n )
 e^{-\de r}  dn dr d\nu_x(k(\xi)). $$

Define functions ${\mathfrak R}_\e, {\mathfrak R}_\e^+,
{\mathfrak R}_\e^- $ on $G$: for $h=a_rnk\in ANK$,
$${\mathfrak R}_\e (h)=e^{-\delta_\G r}\chi_{V_\e}(k) ,\;
{\mathfrak R}_\e^+ (h)=e^{-\delta_\G r}\chi_{V_\e^+}(k) ,\;
{\mathfrak R}_\e (h)=e^{-\delta_\G r}\chi_{ V_\e ^-}(k) .$$
 Note that
$$\int_{B}\psi_\e(b^{-1}) d\rho(b)
=\int_{AN}\psi_\e( ga_tng^{-1} )e^{-(n-1)t} dtdn$$
and hence
$$ e^{-(n-1)\e}\le  \int_{B}\psi_\e(b^{-1})d\rho(b)\le  e^{(n-1)\e} .$$

We then have
\begin{align*}
& m_\G^{\BR}(F^{+}_{\eta}
 *\chi_{V_\e}) =\widetilde{m}^{\mathrm{BR}}(f^{+}
_{\eta}*\chi_{V_\e})
\\ & =\int_{KAN}
 \int_{k_1\in V_{\e}} f^{+}_{\eta} (k (a_r n k_1))
\chi_{V_\e}(k_1) e^{-\de r} dk_1 dn dr d\nu_x(k(\xi))\\
 &=  \int_{k\in K}\int_{h\in G}
f^{+}_{\eta}(k h ) {\mathfrak R}_\e (h) dh  d\nu_x(k(\xi)) \\
 &=  \int_{k\in K}\int_{h\in G}
f^{+}_{\eta}(h ) {\mathfrak R}_\e (k^{-1} h) dh  d\nu_x(k(\xi)) \\
&= \int_{k\in K}\int_{k_2\in K^g}\int_{b\in B}
f^{\pm}_{\eta}(k_2 b^{-1}g ) {\mathfrak R}_\e (k^{-1} k_2 b^{-1}g)
 d\rho(b)dk_2 d\nu_x(k(\xi))  \\
&= \int_{k\in K}\int_{k_2\in K^g}\int_{b\in B}
\phi_{\Om_2,\e(\eta)}^{+}(k_2)\psi_{\e(\eta)}( b^{-1} ) {\mathfrak R}_\e (k^{-1} k_2 b^{-1}g)
 d\rho(b)dk_2 d\nu_x(k(\xi))  \\
&=(1+O(\eta))\int_{k_2\in K^g}\int_{k\in K}
\phi^{+}_{\Om_2, \e(\eta)}(k_2 ) {\mathfrak R}_\e^+ (k^{-1}k_2g)
 dk_2 d\nu_x(k(\xi))
.\end{align*}

For $h\in G$,
define $\hat k_{h}\in K$ to be the unique element
such that
$$h\in B_0 \hat k_{h}.$$
We note that
$$\hat k_{k^{-1}k_2g} =\hat k_{k^{-1}g}(g^{-1}k_2g) .$$

Hence together with Lemma \ref{r},
$${\mathfrak R}_\e^{\pm} (k^{-1}k_2 g)=\chi_{V_\e^{\pm}} (\hat k_{k^{-1}g}(g^{-1}k_2g))
\cdot \e^{-\delta_\G \beta_{k\xi}(y,x)} .$$

Define functions $\tilde \phi_{\Omega_2,\e}^{\pm}\in C(K^g)^{M^g}$ by
$$\tilde \phi_{\Omega_2,\e}^{+}(k_2) :=\sup_{k\in U_\e^+}\phi_{\Omega_2,\e}^+(k_2k)\quad\text{and} \quad \tilde \phi_{\Omega_2,\e}^{-}(k_2) :
=\inf_{k\in  U_\e^-}\phi_{\Omega_2,\e}^-(k_2k).$$
Note that $0\le \tilde \phi^+_{\Om_2, \e}\le 1$ vanishes outside $\Omega_2 U_\e^+$
and is $1$ on $\Omega_2$.

Therefore, using the conformal property of $\{\nu_x:x\in \bH^n\}$:
$$e^{-\delta_\G\beta_{k\xi}(y,x)}d\nu_{x}(k\xi)=
d\nu_y(k\xi),$$
we have
\begin{align*}
& \int_{k_2\in K^g}\int_{k\in K}
\phi^{+}_{\Om_2, \e (\eta)}(k_2 ) \chi_{V_\e^+} (\hat k_{k^{-1}g}(g^{-1}k_2g))
\cdot e^{-\delta_\G \beta_{k\xi}(y,x)}
dk_2 d\nu_x(k(\xi))\\&=
\int_{k_2\in K^g}\int_{k\in K}
\phi^{+}_{\Om_2, \e (\eta)}( g\hat k_{k^{-1}g} ^{-1}g^{-1} k_2 ) \chi_{V_\e^+}
(g^{-1}k_2g)
dk_2 d\nu_y(k(\xi))\\ &\le
  \int_{k_2\in K^g}\int_{k\in K}
\tilde \phi^{+}_{\Om_2, \e (\eta)}(g\hat k_{k^{-1}g} ^{-1} g^{-1} )
 \chi_{V_\e^+} (g^{-1}k_2g)
dk_2 d\nu_y(k(\xi))\\
&= (1+O(\eta))\op{Vol}(V_\e) \int_{k\in K}
\tilde \phi^{+}_{\Om_2, \e (\eta)}(g\hat k_{k^{-1}g} ^{-1} g^{-1} ) d\nu_y(k(\xi)).
\end{align*}

Since $ kB_0= (g\hat k_{k^{-1}g} ^{-1} g^{-1})(gB_0)$
and $B_0$ stabilizes $\xi$,
we have \begin{equation}\label{xi}
k(\xi)=(g\hat k_{k^{-1}g} ^{-1} g^{-1})(g\xi).\end{equation}
Therefore we have
\begin{align*} &m^{\BR}_\G(F^{+}_{\eta}
 *\chi_{V_\e})
\\ &=
 (1+O(\eta)) \op{Vol}(V_\e) \int_{k'\in K^g}
\tilde \phi^{+}_{\Om_2, \e (\eta)}(k' ) d\nu_y(k'(g(\xi)))
\\ &= (1+O(\eta))
 \op{Vol}(V_\e)  \nu_y(\Om_2(g(\xi)))
\quad\text{ by \eqref{ue}}.
\end{align*}

Hence we conclude
$$\limsup_\e \frac{m_\G^{\BR}(F^+_{\eta} *\chi_{V_\e})}
{\op{Vol}(V_\e)} =(1+O(\eta)) \nu_y(\Om_2(g(\xi))) .$$

Similarly we can deduce
$$\liminf_{\e} \frac{m_\G^{\BR}(F^-_{\eta} *\chi_{V_\e})}{\op{Vol}(V_\e)}
=(1+O(\eta))  \nu_y(\Om_2(g(\xi))).$$

On the other hand, it is not hard to deduce from the continuity of $F^{\pm}_\eta$
that $$m^{\BR}_\G (F^{\pm}_\eta)=\lim_\e  \frac{m^{\BR}_\G(F^\pm_{\eta} *\chi_{V_\e})}
{\op{Vol}(V_\e)}.$$

Hence
$\limsup_\eta m_\G^{\BR}(F^+_\eta)=\nu_y(\Om_2(g(\xi)))=\liminf_\eta m_\G^{\BR}(F^-_\eta).$
 \end{proof}

\noindent{\bf Proof of Theorem
 \ref{thm:main}.}

  Since $\nu_x(\partial(\Om_1))=0$
and any circle with center in $\Lambda(\G)$ has measure zero by Lemma \ref{rud},
we may choose $V_\e$ so that
$\nu_x(\partial(\Om_{1,\e}^+(\xi_x)))=
\nu_x(\partial(\Om_{1,\e}^-(\xi_x)))=0$.

By Lemma \ref{lem:thicken1}, Theorem \ref{thm:6.1} and Lemma \ref{aux}, we have
\begin{align*}&\limsup_{T} \frac{N_T^C(\hat\Om_1,\hat\Om_2)}{e^{\delta_\G T}}\le \limsup_{T,\eta}
\frac{1}{e^{\delta_\G T}} \int_{\Pb_0(T+\e,\Om_{1,\e}^+)} F^{+}_{\eta}(b_0)
d\rho_\ell (\q_0) \\
&= \limsup_{\eta} \frac{(1+O(\eta))
\nu_x (\Om_1 (\xi_x))}{\delta_\G \cdot |\bms |}\cdot
\limsup_\eta
 m_\G^{\BR}(F^+_{\eta} )
\\
&=\frac{\nu_x  (\hat \Om_1 ) \nu_y(\hat \Om_2)}{\delta_\G \cdot |\bms |}.
  \end{align*}

Similarly,
\begin{align*}&\liminf_{T,} \frac{N_T(\hat\Om_1,\hat\Om_2)}{e^{\delta_\G T}}\ge \liminf_{T,\eta}
\frac{1}{e^{\delta_\G T}} \int_{\Pb_0(T-\e ,\Om_{1,\e}^-)} F^{-}_{\eta}(b_0)d\rho_{\ell}(\q_0) \\
&= \liminf_{\eta}
\frac{(1+O(\eta))\nu_x (\Om_1 (\xi_x))}{\delta_\G \cdot |\bms |}\cdot
\liminf_\eta  m_\G^{\BR}(F^-_{\eta} )
\\
&=\frac{\nu_x (\hat \Om_1 ) \nu_y(\hat \Om_2) }{\delta_\G \cdot |\bms |}
  .\end{align*}

Since $|N_T-N_T(C)|\le \# \G\cap K\{a_t:0\le t\le C\} K$
is a finite number independent of $T$,
the above proves that
 $$
N_T(\hat \Omega_1, \hat \Omega_2)\sim
{e^{\de T}} \cdot \frac{\nu_x(\hat \Om_1 ) \nu_y(\hat \Om_2) }{\de_\G \cdot |\bms|}
 .$$

\qed



\newcommand{\goth}[1]{\EuFrak{#1}}
\renewcommand{\C}{{\bf C}}
\renewcommand{\R}{{\bf R}}
\renewcommand{\Q}{{\bf Q}}
\newcommand{\Z}{{\bf Z}}
\newcommand{\Bsl}{{\bf SL}}
\newcommand{\NN}{{\bf N}}

\newtheorem{thm}{Theorem}
\newtheorem{conj}{Conjecture}
\newtheorem{lem}{Lemma}[section]
\newtheorem{prop}{Proposition} 
\newtheorem{cor}{Corollary}[section]
\newtheorem{df}{Definition}[section]



\newpage
\renewcommand{\thesection}{A}

\section{Appendix by Fran\c cois Maucourant}

\label{app}


\begin{abstract}
 The purpose of this note is to show how one can recover a result in the spirit of Lim and Oh from a Theorem of Roblin. The following is part of the author's PhD Thesis \cite{Mau}, with some minor modifications, and some of these ideas have also been used in \cite{GorodnikMaucourant}, but in the case of lattices in higher rank Lie groups.
\end{abstract}
\maketitle


 Let $(X,d)$ be a CAT(-1) space, and $\Gamma$ a discrete, non-elementary subgroup of isometries of $X$. Denote by $\partial X$ the visual boundary of $X$, $\overline{X}=X\cup \partial X$, $\delta$ the critical exponent of $\Gamma$, which is assumed finite, $\{\nu_x\}$ the Patterson-Sullivan density for $\Gamma$, and $m_\Gamma^{\BMS}$ the associated
  Bowen-Margulis-Sullivan measure. We shall assume that the length spectrum is non-arithmetic, and that $m_\Gamma^{\BMS}$ is of finite mass; remark that all these hypotheses are satisfied in the case of geometrically finite groups on hyperbolic spaces. First, let us state Roblin's Theorem.
\begin{thm} \cite[Theorem 4.1.1]{Roblin2003}  Let $f$ a continuous function from $\overline{X}^2$ to $\R$, and $(x,y)\in X^2$. Then
$$\lim_{T\rightarrow +\infty} \frac{\delta ||m_\Gamma^{\BMS}||}{e^{\delta T}} \sum_{\gamma \in \Gamma, d(x,\gamma y)\leq T} f(\gamma y,\gamma^{-1} x)=\int_{\partial X^2} f(\xi,\eta) d\nu_x(\xi)d\nu_y(\eta).$$
\end{thm}

 We shall prove here:

\begin{thm}\label{mmmm} Let $f$ be a continuous function on $\overline{X}^2$, and $\zeta \in \partial X$.
$$\lim_{T\rightarrow +\infty} \frac{\delta ||m_\Gamma^{\BMS}||}{e^{\delta T}} \sum_{\gamma \in \Gamma, d(x,\gamma y)\leq T} f(\gamma \zeta,\gamma^{-1} x)=\int_{\partial X^2} f(\xi,\eta) d\nu_x(\xi)d\nu_y(\eta).$$
\end{thm}

 A simplified version (where $f$ does not depend on the second coordinate) appeared in \cite{Mau}. The argument divides into three steps: first, we show a quantitative estimate for the recurrence of the action of $\Gamma$ on the set of geodesics, of independent interest. Second, we show that the quantity on the left-hand side above does not depend too much on $\zeta$. Third, we integrate over $\zeta$ to be able to apply Roblin's Theorem.

\bigskip

\noindent{\bf 1: At most linear recurrence on the set of geodesics}
\bigskip

 Define $\mathcal{G}=(\partial X)^2-diag$ to be the set of bi-infinite oriented geodesics on $X$, and let $SX$ be the set of isometric embedding of $\R$ to $X$. The geodesic flow
 $(g^t)_{t \in \R}$ is the time-shift $g^tf(s)=f(s+t)$, and the canonical projection $\pi : SX \rightarrow X$ is the map $\pi(f)=f(0)$.
 We shall make the usual identification
 $$SX=\mathcal{G} \times \R,$$
and this can be done in such a way that $\pi((\xi,\eta),0))$ is the point of the geodesic from $\xi$ to $\eta$ closest to a fixed reference point $o \in X$. In such coordinates, the geodesic flow is just
$g^t((\xi,\eta),s)=((\xi,\eta),s+t)$, whereas the action of $\Gamma$ on $SX$ defines a cocyle $c: \Gamma \times \mathcal{G}\rightarrow \R$, such that for any $((\xi,\eta),t) \in SX=\mathcal{G} \times \R$, we have
 $$\gamma((\xi,\eta),t)=((\gamma \xi,\gamma \eta), t+c(\gamma,(\xi,\eta)) ).$$

Note that $|c(\gamma,(\xi,\eta))|$ is the distance between the projections of $o$ and $\gamma^{-1}o$ on the geodesic from $\xi$ to $\eta$, and recall (see \cite[Corollary 5.6]{Bal}) that in CAT(0) spaces, projection on a closed convex set is uniquely defined and $1$-Lipschitz, so the following inequality holds for any $\gamma \in \Gamma$, $(\xi,\eta) \in \mathcal{G}$:
$$|c(\gamma,(\xi,\eta))|\leq d(o,\gamma o).$$

\begin{prop}
Let $K$ be a compact subset of $\mathcal{G}$, and $(x,y) \in X^2$. Then there exists $C_{K}>0$ and $T_{x,y}>0$ such that for any $(\xi,\eta) \in \mathcal{G}$, and any $T\geq T_{x,y}$,
$$|\{ \gamma \in \Gamma \, : \, d(\gamma x,y) \leq T, (\gamma \xi, \gamma \eta) \in K \}| \leq C_{K} T.$$
\end{prop}
\begin{proof}
 For $v \in \Gamma\backslash SX$, define
 $$f(\Gamma v)=\sum_{\gamma \in \Gamma} 1_{K\times [0,1]}(\gamma v).$$
 Since $\Gamma$ is discrete and acts properly on $SX$, and $K\times [0,1]$ is a compact subset of $SX$, it follows that $f$ is uniformly bounded by some constant $C_0$ depending only on $K$.
 Choose $T_{x,y}=d(o,x)+d(o,y)+1$, $T\geq T_{x,y}$ and let $\gamma \in \Gamma$ such that $(\gamma \xi,\gamma \eta) \in K$, and $d(y, \gamma x)\leq T$.
 Define $v=((\xi,\eta),0)$, then $g^{-c(\gamma,(\xi,\eta))}(\gamma v) \in K\times \{0\}$. Thus,

 $$\int_{-c(\gamma,(\xi,\eta))}^{-c(\gamma,(\xi,\eta))+1} 1_{K\times [0,1]}(\gamma g^t v)dt=1,$$
 and so, since $|c(\gamma,(\xi,\eta))|\leq d(o,\gamma o) \leq d(x, \gamma y) + d(o,x)+d(o,y)\leq T+T_{x,y}-1$,
 $$1\leq \int_{-T-T_{x,y}}^{T+T_{x,y}} 1_{K\times [0,1]}(\gamma g^t v)dt.$$
 Summing over all such $\gamma$, we obtain
$$ |\{ \gamma \in \Gamma \, : \, d(\gamma x,y) \leq T, (\gamma \xi, \gamma \eta) \in K \}| \leq
\int_{-T-T_{x,y}}^{T+T_{x,y}} f(g^t \Gamma v) dt,$$
and the right hand side is bounded by $2(T+T_{x,y})C_0 \leq 4C_0T$.
\end{proof}

\bigskip
\noindent{\bf 2: Second and third steps}
\bigskip

 Let $f$ be a continuous function on $\overline{X}^2$. Define
 $$F(\zeta,x,T)=\frac{1}{|\Gamma x\cap B_T(y)|} \sum_{\gamma \in \Gamma, d(x,\gamma y)\leq T} f(\gamma \zeta,\gamma^{-1} x).$$
 Let $\epsilon>0$, then since $f$ is uniformly continuous, there exists a neighborhood $U$ of the diagonal in $\partial X^2$ such that for any $(\xi,\eta)\in U$ and any $z \in X$, $|f(\xi,z)-f(\eta,z)|\leq \epsilon$. Let $K$ be the complement of $U$, which is a compact subset of $\mathcal{G}$. So
 \begin{multline*}
  |F(\xi,x,T)-F(\eta,x,T)|\leq \frac{1}{|\Gamma x \cap B_T(y)|} \\
  *\left(  \sum_{\gamma \in \Gamma, d(x,\gamma y)\leq T, (\gamma \xi, \gamma \eta) \in U} \epsilon +
   \sum_{\gamma \in \Gamma, d(x,\gamma y)\leq T, (\gamma \xi, \gamma \eta) \in K} 2||f||_\infty \right),
\end{multline*}
 By Proposition 1, the last sum contains at most $O(T)$ terms, so for sufficiently large $T$,
 $$|F(\xi,x,T)-F(\eta,x,T)|\leq 2\epsilon.$$
 This proves that $F(\zeta,x,T)$ does not depend too much on $\zeta$ for large $T$, so for any $\zeta$, its value is close to the integral with respect to any probability measure. Fix $y$, it will then be sufficient to prove that the function
 $$L(T,x,y)=\int_{\partial X} F(\zeta,x,T)\frac{d\nu_{y}(\zeta)}{||\nu_y||},$$
 has limit $\frac1{||\nu_x||.||\nu_y||}\int fd\nu_x\nu_y$ as $T\rightarrow +\infty$; indeed, recall \cite{Roblin2003} that the orbital function satisfies
 $$|\Gamma x \cap B_T(y)|\sim \frac{||\nu_x||.||\nu_y||}{\delta ||m_\Gamma^{\BMS}||}e^{\delta T}.$$

  Define the map $g$ for any $z \in \Gamma y$ and any $x \in \overline{X}$ by:
 $$g(z,x)=\frac{1}{||\nu_y||}\int_{\partial X} f(\zeta,x)d\nu_z(\zeta),$$
 and extend $g$ when $z$ is in the limit set $\Lambda_\Gamma$, to be equal to $f(z,x)$. Then $g$ is continuous on $\overline{\Gamma y}\times \overline{X}$. By Tietze-Urysohn's Theorem, $g$ can be extended to a continuous function, still denoted by $g$, on $\overline{X}^2$, and moreover $\int gd\nu_xd\nu_y=\int fd\nu_xd\nu_y$.
 Then
 $$L(T,x,y)=\frac{1}{|\Gamma x\cap B_T(y)|.||\nu_y||} \sum_{\gamma \in \Gamma, d(x,\gamma y)\leq T} \int_{\partial X} f(\zeta,\gamma^{-1} x)d\nu_{\gamma y}(\zeta),$$
 $$=\frac{1}{|\Gamma x\cap B_T(y)|} \sum_{\gamma \in \Gamma, d(x,\gamma y)\leq T} g(\gamma y, \gamma^{-1}x),$$
 and by Roblin's Theorem applied to the function $g$, we conclude that $L(T,x,y)$ has limit $\frac1{||\nu_x||.||\nu_y||}\int fd\nu_xd\nu_y$ as $T\rightarrow +\infty$, as desired.\\

 {\it Acknowledgments}. Maucourant wishes to thank Thomas Roblin for suggesting improvements on the hypotheses.

\bibliographystyle{plain}

\end{document}